\documentclass{amsart}
\makeatletter

\makeatletter
\def\LaTeX{\leavevmode L\raise.42ex
\hbox{\kern-.3em\size{\sf@size}{0pt}\selectfont A}\kern-.15em\TeX}
\makeatother

\newcommand{\BibTeX}{{\rm B\kern-.05em{\sci\kern-.025emb}\kern-.08em\TeX}}

\makeatletter
\def\@currentlabel{2.1}\label{e:dispaa}
\def\@currentlabel{2.21}\label{e:dispau}
\def\@currentlabel{2.22}\label{e:dispav}
\def\@currentlabel{2.23}\label{e:dispaw}
\def\@currentlabel{2.24}\label{e:dispax}
\def\theequation{\thesection.\@arabic\c@equation}
\makeatother

\def\theequation{\arabic{section}.\arabic{equation}}

\newtheorem{theorem}{Theorem}[section]

\newtheorem{lemma}{Lemma}[section]

\newcommand{\ep}{\epsilon}
\newcommand{\htu}{\hat{u}}
\newcommand{\R}{{\mathbb R}}
\newcommand{\Om}{\Omega}

\title[Asymptotic Behavior]{Asymptotic behavior of a fourth order mean field equation with Dirichlet boundary condition}

\author[F.Robert]{Fr\'ed\'eric Robert}
\address{\noindent F. Robert -- Laboratoire J.A.Dieudonn\'e,
Universit\'e de Nice-Sophia Antipolis, Parc Valrose, 06108 Nice Cedex 2,
France}
\email{frobert@math.unice.fr}
\author[J.Wei]{Juncheng Wei}
\address{\noindent J. Wei -- Department of Mathematics,
Chinese University of Hong Kong, Shatin, Hong Kong}
\email{wei@math.cuhk.hk}

\begin{document}

\begin{abstract}
We consider asymptotic behavior of the following fourth order equation
\[ \Delta^2 u= \rho \frac{e^{u}}{\int_\Om e^{u}\, dx} \  \ \mbox{in} \ \Om, \ u= \partial_\nu u=0 \ \mbox{on} \ \partial \Omega
\]
where $\Om$ is a smooth oriented bounded domain in $\R^4$.
Assuming that $0<\rho \leq C$, we completely characterize the
asymptotic behavior of the unbounded solutions.

\end{abstract}

\date{August 26th 2007.}
\maketitle

\section{Introduction}
\setcounter{equation}{0} In this paper, we study the asymptotic
behavior of unbounded solutions for the following fourth order
mean field equation under Dirichlet boundary condition
\begin{equation}
\label{1.1}
\left\{\begin{array}{l}
\Delta^2 u= \rho \frac{ e^{u}}{ \int_\Om  e^{u}\, dx} \ \mbox{in} \ \Om,\\
u=\frac{\partial u}{\partial \nu}=0 \ \ \mbox{on} \ \partial \Om
\end{array}
\right.
\end{equation}
where $\rho>0$ and $\Omega \subset \R^4$ is a smooth oriented bounded domain.
In dimension two, the analogous problem
\begin{equation}
\label{1.2}
\left\{\begin{array}{l}
-\Delta u= \rho \frac{ e^{u}}{ \int_\Sigma  e^{u}\, dx} \ \mbox{in} \ \Sigma,\\
u=0 \ \ \mbox{on} \ \partial \Sigma
\end{array}
\right.
\end{equation}
where $\Sigma$ is a smooth bounded domain in $\R^2$, has been
extensively studied  by many authors.  Let $ (u_k, \rho_k)$ be a
unbounded  sequence of solutions  to (\ref{1.2}) with $ \rho_k
\leq C, \max_{x \in \Sigma} u_k (x) \to +\infty$. Then it has been
proved that

\medskip(P1) ({\bf no boundary bubbles}) $u_k$ is uniformly bounded near a
neighborhood of $\partial \Sigma$ (Nagasaki-Suzuki \cite{ns},
Ma-Wei \cite{mw});

\smallskip(P2) ({\bf bubbles are simple}) $\rho_k \to 8m \pi$ for some $ m
\geq 1$ and\par $ u_k (x) \to 8 \pi \sum_{j=1}^m G (\cdot,  x_j)$
in $C^2_{loc} (\Sigma \backslash \{x_1, ..., x_m\})$  (Br\'ezis-Merle
\cite{bm}, Li-Shafrir \cite{ls}, Nagasaki-Suzuki \cite{ns}, Ma-Wei
 \cite{mw}), where $G$ is the Green function of $-\Delta$ with Dirichlet boundary condition. Furthermore, it holds that
\begin{equation}
\label{mc}
   \nabla_{x}  R(x_j, x_j) + \sum_{i \not = j} \nabla_x G(x_i, x_j)=0, j=1,..., m
\end{equation}
where $R(x,y)= G(x, y)-\frac{1}{2 \pi} \log \frac{1}{|x-y|}$ is the regular part of $G(x, y)$.

\medskip\noindent On the other hand, giving $m$ points satisfying (\ref{mc}),
Baraket and Pacard \cite{bp} constructed multiple bubbling
solutions to (\ref{1.2}) when  the bubble points satisfy
nondegeneracy condition. Del Pino, Kowalczyk and Musso
\cite{dkm}  constructed multiple bubbling solutions to
(\ref{1.2}) when  the bubble points are {\it topologically
nontrivial}.  Li \cite{li} initiated the computation of Leray-Schauder degree of the solutions to (\ref{1.2}). He  showed in \cite{li} that the Leray-Schauder degree
 remains a constant for $\rho \in (8\pi (m-1), 8\pi m)$ and that the  degree depends only on the Euler number of the domain.  Chen and Lin \cite{cl1,cl2}  obtained the sharp estimates for the bubbling rate and the exact Leray-Schauder degree counting formula of all solutions to (\ref{1.2}) for all
$\rho \not\in 8\pi\mathbb{N}$. A related question connected to physics consists in adding
Dirac masses to the nonlinear parts: we refer to Bartolucci-Chen-Lin-Tarantello \cite{bclt} and to Tarantello
\cite{tarantello} for results and asymptotics in this context. \par

\medskip\noindent In \cite{w}, the second author considered the following
fourth order equation under Navier boundary condition
\begin{equation}
\label{1.3}
\left\{\begin{array}{l}
\Delta^2 u= \rho \frac{ e^{u}}{ \int_\Om  e^{u}\, dx} \ \mbox{in} \ \Om,\\
u=\Delta u=0 \ \ \mbox{on} \ \partial \Om
\end{array}
\right.
\end{equation}
where $\Omega \subset \R^4$ is a smooth and bounded domain.
Assuming that $\Omega$ is convex, the corresponding property (P1)
and (P2) are established in \cite{w}. Later, Lin and Wei \cite{lw}
considered the attainment of least energy solution and removed the
convexity assumption of \cite{w}. Therefore, property (P1) and
(P2) are established for (\ref{1.3}). Sharp estimates for the
bubbles and the computation of topological degree are contained in
\cite{lw1} and \cite{lw2}.\par
\medskip\noindent The purpose of this paper is to establish the corresponding
property (P1) and (P2) for equation (\ref{1.1}): indeed, equation \eqref{1.1} is more natural than \eqref{1.3} from the viewpoint of the Adams inequality (see \eqref{ineq:adams} below). Our main
result can be stated as follows.

\begin{theorem}
\label{main}
 Assume that $\Omega $ is a bounded smooth domain in $\R^4$. Let $(u_k, \rho_k)$ be a sequence of solutions to (\ref{1.1}) such that
\begin{equation}
0<\rho_k \leq C, \max_{x \in \Omega} u_k (x) \to +\infty.
\end{equation}
Then

(a) $ \rho_k \to 64 \pi^2  m$ for some positive integer $m$.

(b) $u_{k} $ has $m-$point blow up, i.e., there exists a set $ S =
\{ x_{1}, ..., x_{m} \}  \subset \Omega $  such that $ \{ u_{k }
\} $ have a limit $u_0 (x)$ for $ x \in \overline{\Omega}
\backslash S $, where the limit function $ u_{0} (x) $ has the
form
\begin{equation}
u_{0} (x) = 64 \pi^2  \sum_{i=1}^{m} G ( x, x_{j} )
\end{equation}
where   $G (x, y ) $ denotes the Green's function of $ \Delta^{2}
$ under the Dirichlet condition, that is
\begin{equation}
\label{green}
 \Delta^{2} G (x, y) = \delta (x-y) \ \mbox{in} \ \Om, \ G(x,y) =\partial_\nu G(x,y) =0 \ \mbox{on} \ \partial \Omega .
\end{equation}

Furthermore, blow up points $x_{j} \in \Omega \ ( 1 \leq j \leq m
)$ satisfy the following relation
\begin{equation}
\label{id} \nabla_{x } R ( x_{j}, x_{j} ) + \sum_{ l \not = j }
\nabla_{x } G ( x_{j}, x_{l} ) =0 \ ( 1 \leq j \leq m )
\end{equation}
where
\begin{equation}
\label{robin}
R (x, y) = G (x, y) + \frac{ \log | x-y|}{ 8 \pi^2  }.
\end{equation}
\end{theorem}

The main difficulty (and main difference) between (\ref{1.1}) and
(\ref{1.3}) is  that for fourth order equations, Maximum Principle
works for Navier boundary conditions but doesn't work for
Dirichlet boundary conditions. More precisely, Green's function
for the Navier boundary condition
\begin{equation}
\label{green1}
 \Delta^{2} G (x, y) = \delta (x-y) \ \mbox{in} \ \Om, \ G(x,y) =\Delta  G(x,y) =0 \ \mbox{on} \ \partial \Omega
\end{equation}
is positive but the Green's function for Dirichlet boundary
condition may become negative (see \cite{ds} and \cite{gs}).   This
poses a major difficulty in using the method of moving planes (as in \cite{lw}) to
exclude the boundary bubbles.  We overcome this by using the Pohozaev identity and by proving strong pointwise estimates for blowing-up solutions to \eqref{1.1}.

\medskip\noindent As an application  of Theorem \ref{main}, we consider the
following minimization problem
\begin{equation}
\label{min}
J_\rho (u)= \frac{1}{2} \int_\Om |\Delta u|^2\, dx - \rho \log
\int_\Om e^u\, dx,
\end{equation}
where $\Om$ is a bounded and smooth domain of $\R^{4}$ and $u\in
H_0^2 (\Omega)$. Here, $H_0^2(\Omega)$ denotes the completion of
$C^\infty_c(\Omega)$ for the norm $u\mapsto \Vert\Delta u\Vert_2$. Adams's version of the Moser-Trudinger inequality \cite{adams} asserts that there exists $C(\Omega)>0$ such that
\begin{equation}\label{ineq:adams}
\displaystyle{\int_\Omega e^{32\pi^2 u^2}}\, dx\leq C(\Omega)
\end{equation}
for all $u\in H_0^2(\Omega)$ such that $\Vert\Delta u\Vert_2=1$. It follows from \eqref{ineq:adams} that $J_\rho$ is bounded from below if and only if $\rho\leq  64
\pi^2$ (for the proof, see the appendix of \cite{lw}).
Furthermore, if $\rho<64 \pi^2 $, the minimizer of $J_\rho$
actually exists, that is, there exists a $u_\rho \in H_0^2 (\Om)$
such that
\begin{equation}
\label{eq:1-2}
J_\rho (u_\rho) := \inf_{u\in H_0^2 (\Om) } J_\rho(u) :=c_\rho.
\end{equation}
For $J_{64 \pi^2}$, it is an interesting question to ask whether
the minimum $c_{64 \pi^2}$ can be attained  or not. The
Euler-Lagrange equation of $J_\rho$ is just (\ref{1.1}). For the
corresponding problem in two dimension, given $\Sigma$ a smooth
two-dimensional domain, we consider
$$E_\rho(u)=\frac{1}{ 2} \int_\Sigma |\nabla u|^2\, dx - \rho \log \left( \int_\Sigma e^u
dx\right), u \in H_0^1 (\Sigma)$$ where $H_0^1 (\Sigma)$
denotes the completion of $C^\infty_c(\Sigma)$ for the norm
$u\mapsto \Vert\nabla u\Vert_2$. Again, by the Moser-Trudinger
inequality, $E_\rho$ is bounded from below if and only if
$\rho\leq 8\pi$, and moreover, the minimum of $E_\rho$ is always
attained if $\rho<8\pi$. However, it has been noted that
minimizers do not always exist for $E_{8\pi}$. Actually, it
depends on the geometry of $\Sigma$ in a very subtle way. For
example, the minimum of $E_{8\pi}$ is not attained if
$\Sigma$ is a ball in $\R^2$, but, it is attained if $\Sigma$ is a
long and thin domain, see \cite{cc1}. So, it is rather surprising
to have the following claim.

\begin{theorem}\label{main2}
Let $\Om$ be a bounded $C^{4}$ domain in $\R^{4}$, and
$u_\rho$ denote a minimizer
of $J_\rho$ for $\rho<64 \pi^2$. Assume that
\begin{equation}
\label{con}
R_1(Q_0, Q_0) + 16 \pi^2  \Delta_x R(Q_0, Q_0)  >0
\end{equation}
for $Q_0\in\Omega$ such that $R(Q_0, Q_0)= \max_{P \in \Omega} R(P, P)$, where $R_1 (x, P)$ is defined by
\begin{equation}
\label{2.12}
\left\{\begin{array}{l}
\Delta^2 R_1 (x, P) =0 \ \mbox{in} \ \Om, \\
 R_1 (x, P) =\frac{4}{|x-P|^2}, \partial_\nu  R_1 (x, P) = \partial_\nu (\frac{4}{  |x-P|^2}),  \ \mbox{on} \ \partial \Om.
\end{array}
\right.
\end{equation}
Then $u_\rho$ is uniformly bounded in $C^{4}$ as $\rho\uparrow 64
\pi^2 $. Consequently, the minimum of $J_{64 \pi^2 }$ can be
attained. As an example, when $\Omega$ is a ball in $\R^4$, $J_{64
\pi^2}$ is attained.
\end{theorem}

\medskip\noindent It is a natural question to ask whether condition \eqref{con} is satisfied for any domain $\Omega$ of $\mathbb{R}^4$: indeed, this question is closely related to the comparison principle for the bi-harmonic operator with Dirichlet boundary condition. For instance, condition \eqref{con} is satisfied if for any function $u\in C^4(\overline{\Omega})$, we have that
\begin{equation}\label{pm}
\left\{\begin{array}{ll}
\Delta^2 u\geq 0 & \hbox{ in }\Omega\\
u\geq 0\hbox{ and }\partial_\nu u \leq 0 &\hbox{ on }\partial\Omega
\end{array}\right\}\Rightarrow \{\, u>0\hbox{ or }u\equiv 0\}.
\end{equation}
It is remarkable that the comparison principle \eqref{pm} does not hold on any domain: for instance, it is false on some annuli. It also follows from Grunau-Robert \cite{gr} that condition \eqref{con} is satisfied on small perturbations of the ball.

\medskip\noindent Semilinear equations involving exponential nonlinearity and fourth
order elliptic operator appear naturally in conformal geometry and
in particular in prescribing $Q-$curvature on 4-dimensional
Riemannian manifold  $M$ (see e.g. Chang-Yang \cite{cy})
\begin{equation}
\label{Q}
P_g w+ 2 Q_g=2 \tilde{Q}_{g_w} e^{4 w}
\end{equation}
where $P_g$ is the so-called Paneitz operator:
\[ P_g= (\Delta_g)^2 +  \delta \left(\frac{2}{3} R_g I- 2 \mbox{Ric}_g \right)d,\]
 $g_w= e^{2w} g$, $Q_g$ is  $Q-$ curvature under the metric $g$, and $\tilde{Q}_{g_w}$ is the $Q$-curvature under the new metric $g_w$. Integrating (\ref{Q}) over M, we obtain
\[ k_g:= \int_M Q_g\, dv_g=  \int_M (\tilde{Q}_{g_w}) e^{4w}\, dv_g= \int_M \tilde{Q}_{g_w} \, dv_{g_w}\]
and $k_g$ is  conformally invariant (here $dv_g$ denote the Riemannian element of volume). Thus, we can write (\ref{Q}) as
\begin{equation}
\label{Q1}
P_g w+ 2 Q_g=2 k_g \frac{ \tilde{Q}_{g_w} e^{4 w}}{\int_M \tilde{Q}_{g_w} e^{4w}\, dv_g}
\end{equation}
In the special case, where the manifold is the Euclidean space, $P_g= \Delta^2$, and (\ref{Q1}) becomes
\begin{equation}
\label{Q2}
\Delta^2 w= 2 k_g  \frac{ h(x) e^{4 w}}{\int_\Om h(x) e^{4w}\, dx }
\end{equation}
With $ u=2 w,  \rho =4 k_g, h \equiv 1$, we arrive at equation (\ref{1.1}). There is now  an extensive  litterature about this problem. For
instance, we refer to Adimurthi-Robert-Struwe \cite{ars},
Baraket-Dammak-Ouni-Pacard \cite{bdop}, Druet \cite{druet2D},
Druet-Robert \cite{dr}, Hebey-Robert \cite{hr}, Hebey-Robert-Wen
\cite{hrw}, Malchiodi \cite{m}, Malchiodi-Struwe \cite{ms}, Robert \cite{robert}, Robert-Struwe \cite{rs} and the
references therein. Note also that recently, Clapp-Mu\~noz-Musso \cite{clapp} have proved the existence of blowing-up solutions to \eqref{1.3} (that is \eqref{1.1} with Navier boundary condition) with arbitrary number of bubbles provided topological hypothesis on $\Omega$.\par \medskip\noindent Our paper  is organized
as follows. In Section 2, we  present two useful lemmas. Theorem
\ref{main} is proved in Section 3 and Theorem \ref{main2} is
proved in Section 4.\par
\medskip\noindent{\bf Notation:} Throughout this paper, the constant $C$ will denote
various constants which are independent of $\rho$: the value of $C$
might change from one line to the other, and even in the same
line. The equality $B =O(A)$ means that there exists $C>0$ such
that $|B|\leq C A $. All the convergence results are
stated up to the extraction of a subsequence.

\section{ Some preliminaries}
\setcounter{equation}{0} We state two results in this section. The
first one concerns the properties of the Green's function
(\ref{green}). The second one is Pohozaev's identity. Recall that
$G(x,y)$ is defined by (\ref{green}). As we remarked earlier,
in general,  $G(x, y)$ is not positive. We collect properties of
$G$ in the following lemma.

\begin{lemma}
There exists $C>0$ such that for all $x,y\in\Omega$, $x\neq y$, we have that
\begin{equation}
\label{g1}
  |G(x, y)| \leq C \log \left(2+ \frac{1}{ |x-y|}\right)
\end{equation}
\begin{equation}
\label{g2}
 | \nabla^i G(x, y)| \leq C |x-y|^{-i},\,  i \geq 1
\end{equation}
\end{lemma}
\begin{proof} These estimates are originally due to
Krasovski\u{\i} \cite{kraso}. We also refer to Dall'Acqua-Sweers
\cite{ds} and Grunau-Robert \cite{gr}. \end{proof}

\medskip\noindent Next we  state a Pohozaev identity for equation (\ref{1.1}).
\begin{lemma}
 Let $ u\in C^4(\overline{\Omega})$ be a solution of $ \Delta^{2} u = f(u) $ in $ \Omega $. Then we have for any $ y \in \R^{4} $,
\begin{eqnarray*}
    && 4 \int_{\Omega } F(u)\, dx \nonumber
  =   \int_{\partial \Omega } \langle x -y, \nu \rangle F(u)\, d\sigma + \frac{1}{2} \int_{\partial \Omega } v^{2} \langle x - y, \nu \rangle d\sigma + 2 \int_{\partial \Omega } \frac{ \partial u}{ \partial \nu } v \, d\sigma  \nonumber \\
 &  &+    \int_{\partial \Omega } \Biggl( \frac{ \partial v}{\partial \nu } \langle x-y, D u \rangle  + \frac{ \partial u}{\partial \nu } \langle x -y, D v \rangle  - \langle Dv, D u \rangle  \langle x-y, \nu \rangle \Biggr)\, d\sigma \
\end{eqnarray*}
where $ F(u) = \int_{0}^{u} f(s)\, ds , -\Delta u = v $ and $ \nu (x)  $ is the normal outward derivative of x on $\partial \Omega$.

\end{lemma}
\begin{proof} More general version of this formula can be seen, for example in \cite{m1}.
In our case, integrating the identity on $\Omega$
\begin{eqnarray*}
&&\mbox{div} ( (x-y, \nabla v) \nabla u+ (x-y, \nabla u) \nabla v- (\nabla u, \nabla v) (x-y))\\
&&= (x-y, \nabla v) \Delta u + (x-y, \nabla u) \Delta v - 2 (\nabla u, \nabla v)
\end{eqnarray*}
for $ u, v \in C^2 (\bar{\Omega})$, $\nabla =\nabla_x$, and noting that
\[ \mbox{div} ( (x-y) F(u))= f(u) (x-y, \nabla u) + 4 F(u)\]
and
\[ \mbox{div} \left(\frac{1}{2} v^2 (x-y) +2 v \nabla u\right)= v ( \nabla v, x-y) +2 (\nabla u, \nabla v)\]
if $v =-\Delta u$, we get the desired formula.\end{proof}

\section{Proof of Theorem \ref{main}}
\setcounter{equation}{0} Let $ u_k$ be a family of solutions to
problem (\ref{1.1}) such that there exists $\Lambda>0$ such that
\begin{equation}
0<\rho_k \leq \Lambda.
\end{equation}
In this section, we study the asymptotic behavior of unbounded
solutions and prove Theorem \ref{main}. Let
$$ \alpha_k:= \log \left(\frac{\int_\Om e^{ u_k}\, dx}{ \rho_k}\right)\hbox{ and }\hat{u}_k:=  u_k -\alpha_k.$$
Theorem \ref{main} is proved by a series of claims. We first claim
that\par
\medskip\noindent{\bf Claim 1:}  There exists $C\in\mathbb{R}$ such that $\alpha_k \geq C$ for all $k\in\mathbb{N}$.

\begin{proof} Note that $\hat{u}_k$ satisfies
\begin{equation}
\label{2}
\Delta^2 \hat{u}_k= e^{ \hat{u}_k} \ \ \mbox{in} \ \Om, \ \htu_k= -\alpha_k, \ \partial_\nu\htu_k=0 \ \mbox{on}\  \partial  \Omega
\end{equation}
with $ \int_\Om e^{\htu_k }\, dx <C$ for all $k$. So
\begin{equation}
\label{3}
\htu_k (x)= \int_\Om G(x, y) e^{ \htu_k (y)} dy  -\alpha_k
\end{equation}
and hence by  (\ref{g2}),
\begin{eqnarray*}
\int_\Om |\Delta \htu_k (x)|\, dx &\leq &\int_\Om \left(\int_\Om |\Delta_x G(x, y)|  e^{ \htu_k (y)} \,dy \right)\,dx\\
& \leq & C \int_\Om \left(\int_\Om \frac{1}{|x-y|^2}   e^{ \htu_k (y)} \,dy \right)\,dx \leq C.
\end{eqnarray*}
Similarly, integrating \eqref{3}, we get that there exists $C>0$
such that
\begin{equation}\label{bnd:L1}\Vert
\htu_k+\alpha_k\Vert_{L^1(\Omega)}\leq C \end{equation} for all
$k\in\mathbb{N}$. It follows from Theorem 1.2 of \cite{robert}
that there exists $S_1 \subset \Om$, where $S_1$ is at most
finite, such that $ \htu_k \leq C(\omega)$ uniformly in $\omega$
for $\omega \subset \subset \Omega \backslash S_1$. Therefore,
with \eqref{bnd:L1}, we get that $(\alpha_k)$ cannot go to
$-\infty$ when $k\to +\infty$. This proves Claim 1.\end{proof}

\medskip\noindent A consequence is the following proposition that concerns the case when $u_k$ is bounded from above:
\begin{lemma}\label{lem:cv} Let $(u_k, \rho_k)$ be a sequence of solutions to (\ref{1.1}) such that there exists $\Lambda>0$ such that $0<\rho_k\leq\Lambda$. Assume that there exists $C>0$ such that $u_k\leq C$ for all $k\in\mathbb{N}$. Then there exists $u\in C^4(\overline{\Omega})$ such that, up to a subsequence $\lim_{k\to +\infty} u_k=u$.
\end{lemma}
\begin{proof} It follows from the assumption of the lemma and Claim 1 that $\hat{u}_k\leq C_1$ on $\Omega$. It then follows from \eqref{1.1} and \eqref{2} that $(u_k)$ is bounded in $C^3(\overline{\Omega})$. The conclusion follows from elliptic theory.
\end{proof}
\noindent In the sequel, we assume that
\begin{equation}\label{lim:infty}
\max_{x \in \Om} u_k (x) \to +\infty.
\end{equation}
\medskip\noindent Our second claim is an upper bound on the $L^p-$norm of $
\nabla^i \htu_k$:\par
\medskip\noindent {\bf Claim 2:} For all $i=1, 2, 3, p \in (1,
\frac{4}{i})$, there exists $C=C(i, p)$ such that $ \| \nabla^i
\htu_k \|_{L^p (\Omega)} \leq C$.\par
\begin{proof} By Green's
representation formula (\ref{3}) and (\ref{g2}), we have
\[ |\nabla^i \htu_k (x)| \leq  \int_\Om | \nabla^i_x G(x, y)| e^{ \htu_k (y)} \,dy  \]
\[ \leq C \int_\Om \frac{1}{|x-y|^i} e^{ \htu_k} \,dy.\]
Thus for any $\varphi \in C_c^\infty (\R^4)$, we have
\begin{eqnarray*}
&&\int_\Om |\nabla^i \htu_k (x)| \varphi \,dx  \leq  \int_\Om \left(\int_\Om | \nabla^i_x G(x, y)| e^{ \htu_k (y)} \,dy \right) |\varphi (x)|\, dx\\
&&\leq C \int_\Om e^{ \htu_k}\left(\int_\Om |x-y|^{-i} |\varphi (x)|\, dx\right)\, dy\\
&&\leq C  \int_\Om e^{ \htu_k}\| |x-y|^{-i} \|_{L^p (\Om)} \|\varphi\|_{L^q (\Om)}\, dy\\
&& \leq C \| \varphi\|_{L^q (\Om)}
\end{eqnarray*}
where $ \frac{1}{p}+\frac{1}{q}=1$. Here, we used that $\Omega$ is
bounded. By duality, we derive that  $ \| \nabla^i \htu_k \|_{L^p
(\Omega)} \leq C$.\end{proof}

\medskip\noindent The third claim asserts that bubbles must have some distance from the
boundary:\par
\medskip\noindent{\bf Claim 3:}  Let $(x_k)_{k \in {\mathbb N}} \in \Om$
be such that $ u_k (x_k)=\max_\Om u_k$. Let $ \mu_k:= e^{-
\frac{1}{4} \htu_k (x_k)}$. Then $ \lim_{k \to +\infty} \frac{d
(x_k, \partial \Om)}{\mu_k} =+\infty$.
\begin{proof} Suppose otherwise, $ d(x_k, \partial \Om)= O(\mu_k)$. Let $\Om_k:=\frac{\Om -x_k}{\mu_k}$.
Then up to a rotation, we may assume that $\Om_k \to (-\infty,
t_0) \times \R^3$. Let $\tilde{u}_k (x):=\htu_k (x_k +\mu_k x) +
4\log \mu_k$. Note that $\lim_{k\to +\infty}\mu_k=0$ (otherwise $\hat{u}_k$ is bounded from above, and, as in the proof of Lemma \ref{lem:cv}, we get that $(u_k)$ is bounded: a contradiction with \eqref{lim:infty}). Let $R>0$
and $x\in B_R(0)\cap\Om_k$, then we have by the representation
formula (\ref{3}) and (\ref{g2})
\begin{eqnarray*}
|\nabla^i \tilde{u}_k  (x)| &=& | \mu_k^i \nabla^i \htu_k (x_k+\mu_k x)|\\
& =& \mu_k^i \left|\int_\Om \nabla_x^i G (x_k +\mu_k x, y) e^{ \htu_k (y)}\, dy\right|\\
&\leq &C \mu_k^i \left( \int_{ B_{2R \mu_k} (x_k)} \frac{1}{ |x_k + \mu_k x-y|^i} e^{ \htu_k (y)} \,dy \right.\\
&& \left.+ \int_{ \Om_k \backslash B_{2R \mu_k} (x_k)}   \frac{1}{ |x_k + \mu_k x-y|^i} e^{ \htu_k (y)}\, dy
 \right).
\end{eqnarray*}
On $\Om_k \backslash B_{2R \mu_k} (x_k), |x_k +\mu_k x-y| \geq
|y-x_k| -\mu_k |x| \geq R \mu_k, e^{ \htu_k (y)} \leq e^{ \htu_k
(x_k)} = \mu_k^{-4}$. Hence
\[ |\nabla^i \tilde{u}_k (x)|\leq \mu_k^{i-4} \int_{ B_{2R\mu_k} (x_k)} \frac{dy}{|x_k +\mu_k x-y|^i} + C \int_\Om e^{ \htu_k} \,dy \leq C (R).\]
In particular, this implies that $ |\tilde{u}_k (x) -\tilde{u}_k
(0)| \leq C |x|$ for all $x \in B_R (0)$. Now let $ x \in \partial
\Om_k$, we get $ |\htu_k (x_k) +\alpha_k | \leq C $. This gives
\[ 4\log \frac{1}{\mu_k} +\alpha_k = O(1).\]
A contradiction with $\lim_{k\to +\infty}\mu_k=0$ and Claim 1.
Thus $\frac{d(x_k, \partial \Om)}{\mu_k} \to +\infty$.\end{proof}
\medskip\noindent Claim 4 concerns the first bubble:\par
\medskip\noindent{\bf Claim 4:} We have that
$$\lim_{k \to +\infty} \htu_k (x_k +\mu_k
x) + 4 \log \mu_k= -4 \log \left(1+\frac{|x|^2}{ 8 \sqrt{6}}\right)$$ in
$C_{loc}^4 (\R^4)$.
\begin{proof} By Claim 3, we have $\Om_k \to \R^4$. Since $ \tilde{u}_k (x)= \htu_k (x_k +\mu_k x)+4 \log \mu_k$, $ \tilde{u}_k (x) \leq \tilde{u}_k (0)$ and $ \Delta^2 \tilde{u}_k = e^{ \tilde{u}_k}$ in $\Om_k$. Note by Claim 3, $ |\nabla^i \tilde{u}_k (x) | \leq C(R)$, for all $ x \in B_R (0)$. By standard regularity arguments,  $ \tilde{u}_k \to \tilde{u}$ in $C_{loc}^4 (\R^4)$ where $\tilde{u}$ satisfies
\begin{equation}
\label{5}
\Delta^2 \tilde{u}= e^{ \tilde{u}}, \ \tilde{u} (0)=0, \int_{\R^4} e^{ \tilde{u}}\, dx <+\infty.
\end{equation}
Note that solutions to (\ref{5}) are nonunique. To characterize $\tilde{u}$, we compute
\[ \Delta \tilde{u}_k (x)= \int_\Om \mu_k^2 \Delta_x G(x_k+\mu_k x, y) e^{ \htu_k (y)}\, dy\]
and for $x \in B_R (0)$,
\[ \int_{B_R(0)} |\Delta \tilde{u}_k|\, dx\leq C\int_{\Om} e^{ \htu_k (y)} \left(\mu_k^2 \int_{B_R(0)} \frac{dx}{ |x_k +\mu_k x-y|^2}\right)\, dy\]
\[ \leq C R^2 \int_{\Om} e^{ \htu_k (y)} \,dy \leq C R^2.\]
That is, for any $R>0$, we have $ \int_{B_R(0)} |\Delta
\tilde{u}|\, dx \leq C R^2$. It then follows from results of
\cite{lin} and \cite{wx} that $\tilde{u} (x)= -4 \log
\left(1+\frac{|x|^2}{ 8 \sqrt{6}}\right) $.  Moreover, $\int_{B_{R \mu_k}
(x_k)} e^{ \htu_k}\, dx= \int_{B_R (0)} e^{ \tilde{u}_k}\, dx$ and
hence
\begin{equation}\label{eq:1}
\lim_{R \to +\infty} \lim_{k \to +\infty} \int_{B_{R\mu_k } (x_k)} e^{ \htu_k}\, dx = 64 \pi^2.
\end{equation}
\end{proof}

\medskip\noindent We say that the property ${\mathcal H}_p$ holds if there exists $
(x_{k, 1}, ..., x_{k, p}) \in \Om^p$ such that, denoting $\mu_{k,
i}:= e^{- \frac{1}{4} \htu_k (x_{k, i})}$, we have that
\begin{itemize}
\item[(i)] $  \ \ \ \lim_{k \to +\infty} \frac{ |x_{k, i}-x_{k, j}|}{\mu_{k, i}}= +\infty, \ \forall i \not = j,$

\item[(ii)] $ \ \ \ \lim_{k \to +\infty} \frac{d (x_{k,i}, \partial \Om)}{\mu_{k, i}}= +\infty,$

\item[(iii)] $ \ \ \ \lim_{ k \to +\infty} (\htu_k (x_{k, i} + \mu_{k, i} x )+ 4 \log \mu_{k, i} ) = -4 \log (1+\frac{|x|^2}{ 8 \sqrt{6}}) \ \mbox{in} \ \ C_{loc}^4 (\R^4)$.

\end{itemize}
\noindent By Claim 4, ${\mathcal H}_1$ holds.\par
\medskip\noindent{\bf Claim 5:} Assume that ${\mathcal H}_p$ holds. Then
either ${\mathcal H}_{p+1}$ holds, or there exists $C>0$ such that
\begin{equation}
\label{8}
 \inf_{i=1,..., p} \{ |x- x_{k, i}|^4 \} e^{\htu_k (x)} \leq C, \forall x \in \Om.
\end{equation}
\begin{proof} Let $ w_k (x):=\inf_{i=1,..., p} |x-x_{k, i}|^4 e^{\htu_k (x)}$. Assume that $\| w_k \|_{L^\infty (\Om)} \to +\infty$ when $k \to +\infty$.
Let $y_k\in \Om$ be such that $w_k (y_k)=\max_{\Om} w_k$ and
$\gamma_k:= e^{- \frac{1}{4} \htu_k (y_k)}$ and $ v_k (x):= \htu_k
(y_k + \gamma_k x) + 4 \log \gamma_k$. Then $v_k$ satisfies
$\Delta^2 v_k= e^{ v_k}$. Note that $ w_k (y_k)= \inf_{i=1,..., p}
\frac{|y_k-x_{k, i}|^4}{\gamma_k^4} \to +\infty$. Then $ \lim_{k
\to +\infty} \frac{ |y_k -x_{k, i}|}{\gamma_k} \to +\infty$ for
all $i=1,..., p$. Assume that there exists $i$ such that $ y_k
-x_{k, i}= O(\mu_{k, i})$, Then $ y_k= x_{k, i} +\mu_{k, i}
\theta_{k, i}$ and
\[|y_k -x_{k, i}|^4  e^{\htu_k (y_k)} = |\theta_{k, i}|^4  e^{\htu_k ( x_{k, i} +\mu_{k, i} \theta_{k, i}) +4 \log \mu_{k, i}} \to |\theta_{\infty, i}|^4   \frac{ 1}{ (1 + \frac{|\theta_{\infty, i}|^2}{8 \sqrt{6}})^4 }\]
where $\theta_{\infty, i}= \lim_{k \to +\infty} \theta_{k, i}$.
This implies that $w_k (y_k)= O(1)$. A contradiction. Thus $
\frac{|y_k- x_{k, i}|}{\mu_{k, i}} \to +\infty$ for all $
i=1,...,p$.\par
\smallskip\noindent Let $x \in B_R (0)$ and let $\epsilon\in (0,1)$. Then $ w_k (y_k +\gamma_k x) \leq w_k (y_k)$. That is,
$\inf_{i=1,..., p} | y_k -x_{k, i} +\gamma_k x|^4 e^{\htu_k (y_k
+\gamma_k x)}\leq \inf_{i=1,..., p} |y_k -x_{k, i}|^4 e^{\htu_k
(y_k)}$ and so
\[e^{v_k (x)} \leq \frac{\inf_{i=1,..., p} |y_k -x_{k, i}|^4}{ \inf_{i=1,..., p}|y_k-x_{k, i} + \gamma_k x|^4}. \]
\noindent Let $k \geq k(R)$ be such that $ \frac{ |y_k -x_{k,
i}|}{\gamma_k} \geq \frac{R}{\epsilon}$ for all $ i=1,..., p, k
\geq k(R)$. Then for $i=1,..., p$, we have $ |y_k -x_{k, i}
+\gamma_k x| \geq |y_k -x_{k, i}| (1-\epsilon)$ and $
\inf_{i=1,..., p} |y_k -x_{k, i} +\gamma_k x|^4 \geq
\inf_{i=1,..., p} |y_k -x_{k, i}|^4 (1-\epsilon)^4$. This yields
\[ e^{v_k (x)} \leq \frac{1}{(1-\epsilon)^4}, \ x \in B_R (0), k \geq k(R).\]
Similar to Claim 3, we also have that
$$ \lim_{k\to +\infty}\frac{ d(y_k, \partial \Om)}{
\gamma_k }=+\infty\hbox{ and }\lim_{k \to +\infty} v_k (x) = -4
\log \left(1+\frac{|x|^2}{8 \sqrt{6}}\right)$$
in $C_{loc}^4 (\R^4)$. Letting
$ x_{k, p+1}= y_k$, then ${\mathcal H}_{p+1}$ holds.  The claim is
thus proved. \end{proof}
\medskip\noindent{\bf Claim 6:} There exists $N$ such that ${\mathcal H}_N$ holds and there exists $C>0$ such that
\begin{equation}
\label{9}
\inf_{i=1,..., p} |x- x_{k, i}|^4 e^{\htu_k (x)} \leq C, \ \forall x \in \Omega.
\end{equation}
\begin{proof} Otherwise, since ${\mathcal H}_1$ holds, then ${\mathcal H}_p$ holds for all $ p \geq 1$. Given $R>0$, we have $B_{R \mu_{k, i} } (x_{k, i}) \cap B_{R \mu_{k, j}} (x_{k, j}) = \emptyset$ for all $ i \not = j, k \geq k(R)$. Then
\[ \rho_k = \int_\Om e^{ \htu_k}\, dx \geq \int_{\cup_{i=1,..., p} B_{R \mu_{k, i}} (x_i)} = \sum_{i=1}^p \int_{B_{R\mu_{k, i}} (x_{k, i})} e^{  \htu_k (y)}\, dy  \geq 64 \pi^2p  + o(1)_{R}\]
where $\lim_{R\to +\infty}\lim_{k\to +\infty}o(1)_R = 0$. Since $\rho_k \leq \Lambda$, we
derive that $p \leq \Lambda/64\pi^2$ for all $p$: a contradiction. Hence Claim 6 holds.\end{proof}
\medskip\noindent{\bf Claim 7:} For $ p=1,2,3$, there exists $C >0$ such that
\begin{equation}
\label{10}
\inf_{i=1,..., p} |x- x_{k, i}|^p |\nabla^p \htu_k (x)|   \leq C, \ \forall x \in \Omega.
\end{equation}
\begin{proof} By Green's representation formula, we have
$$\nabla^p \htu_k (x)= \int_\Om \nabla_x^p G(x, y) e^{ \htu_k (y)}\, dy.$$
Hence
\begin{equation}
|\nabla^p \htu_k (x)| \leq C \int_\Om |x-y|^{-p} e^{ \htu_k (y)}\, dy.
\end{equation}
Let $ R_k (x):= \inf_{i=1,..., N} |x-x_{k, i}|, \Om_{k, i}= \{ x \in \Om: |x-x_{k, i}|= R_k (x)\}$. Then
\begin{eqnarray*}
\int_{\Om_{k, i}} |x-y|^{-p} e^{ \htu_k (y)}\, dy &=& \int_{\Om_{k, i} \cap B_{\frac{|x-x_{k, i}|}{2}} (x_{k, i}) } |x-y|^{-p} e^{ \htu_k (y)}\, dy\\
&&+
 \int_{\Om_{k, i} \backslash B_{\frac{|x-x_{k, i}|}{2}} (x_{k, i}) }|x-y|^{-p} e^{ \htu_k (y)}\, dy.
\end{eqnarray*}
Note that for $y \in \Om_{k, i} \backslash B_{\frac{|x-x_{k,
i}|}{2}} (x_{k, i})$, $ |x-y|^{-p} e^{ \htu_k (y)} \leq
\frac{C}{|x-y|^p |y-x_{k, i}|^4}$. Then Claim 6 and easy computations show
that
\[ \int_{B_{R} (0) \backslash  B_{ \frac{|x-x_{k, i}|}{2}} (x_{k, i})} \frac{1}{ |x-y|^p |y-x_{k, i}|^4 }\, dy \leq \frac{C}{|x-x_{k, i}|^p}.\]
Thus
\begin{equation}
\label{13}
 \left| \int_{\Om_{k, i} \backslash B_{\frac{|x-x_{k, i}|}{2}} (x_{k, i}) } |x-y|^{-p}e^{ \htu_k (y)}\, dy\right| \leq \frac{C}{|x-x_{k, i}|^p}.
\end{equation}
On the other hand,  for $y \in \Om_{k, i} \cap B_{\frac{|x-x_{k,i}|}{2}} (x_{k, i})$, we have $ |x-y| \geq |x-x_{k, i}| - |y-x_{k, i}| \geq \frac{1}{2} |x-x_{k, i}|$ and hence
\begin{equation}
\label{14}
 \left| \int_{\Om_{k, i} \cap  B_{\frac{|x-x_{k, i}|}{2}} (x_{k, i}) } |x-y|^{-p}e^{ \htu_k (y)}\, dy\right| \leq \frac{C}{|x-x_{k, i}|^p}.
\end{equation}
Combining (\ref{13}) and (\ref{14}), we obtain the desired
estimates. \end{proof}
\medskip\noindent{\bf Claim 8:}
 Let $x_i:= \lim_{k \to+\infty} x_{k, i} \in \bar{\Om}$ and $S:=\{ x_i, i=1,..., N\}$.
 Assume that $\lim_{k\to +\infty}\alpha_k=+\infty$. Then $\htu_k \to -\infty$ uniformly in $\bar{\Om}\backslash S$.
\begin{proof} Let $\delta>0$ small such that $\Om_\delta:=\Om \backslash \cup_{i=1}^N \overline{B}_\delta (x_i)$ is connected. Then $ |\nabla \htu_k (x) | \leq C (\Om_\delta)$ for $x \in \Omega_\delta$ by the representation formula (\ref{3}). Let $ x_\delta \in \partial \Om_\delta \cap \partial \Om$, then we have $\htu_k (x)=-\alpha_k$ and hence $ |\htu_k (x)+\alpha_k | \leq C$ for all $x \in \Om_\delta$. This implies that $ \htu_k \to -\infty$ uniformly.
\end{proof}
\medskip\noindent{\bf Claim 9:} Assume that $\lim_{k\to +\infty}\alpha_k=+\infty$. Then there exists $ \gamma_1, ..., \gamma_{N}
\geq 64 \pi^2 $ such that
$$\lim_{k \to +\infty} u_k (x)=\sum_{i=1}^{N}\gamma_i G(\cdot, x_i)\hbox{ in }C_{loc}^4 (\bar{\Om} \backslash S).$$
\begin{proof} Since $u_k$ satisfies
\[ \Delta^2 u_k= e^{-\alpha_k} e^{ u_k}\]
and $ u_k$ is bounded in $C^0_{loc} (\bar{\Om} \backslash S)$ by Claim 8, by standard regularity arguments we deduce that $ u_k \to \psi$ in $C^4 (\bar{\Om} \backslash S)$, where  $\psi \in C^4 (\bar{\Om} \backslash S)$. Thus,
 for $\delta>0$ small enough,
\[ u_k (x)= \int_\Om G(x, y) e^{ \htu_k (y)}\, dy= \sum_{i=1}^N \int_{B_\delta (x_i) \cap \Om} G(x, y) e^{ \htu_k (y)}\, dy + o(1).
\]
Since $ G(x, \cdot)$ is continuous in $\bar{\Om} \backslash \{x\}$, we get that
\[ \lim_{k \to +\infty} u_k (x)= \sum_{i=1}^N \gamma_i G(x, x_i)\]
where $\gamma_i:= \lim_{\delta\to 0}\lim_{k \to +\infty}
\int_{B_\delta (x_i) \cap \Om} e^{ \htu_k (y)}\, dy$. By Claims 4
and 5, $\gamma_i\geq 64 \pi^2 $. Then $\psi= \sum_{i=1}^N \gamma_i
G(x, x_i)$. So we get the result.\end{proof}
\medskip\noindent{\bf Claim 10:}
 Let $x_i:= \lim_{k \to+\infty} x_{k, i} \in \bar{\Om}$ and $S:=\{ x_i, i=1,..., N\}$.
 Assume that $\lim_{k\to +\infty}\alpha_k=\alpha_\infty\in\mathbb{R}$. Then $S\subset\partial\Omega$ and there exists $u\in C^4(\overline{\Omega})$ such that
 $\Delta^2u=e^{-\alpha_\infty}e^u$ in $\Omega$, $u=\partial_\nu
 u=0$ in $\partial\Omega$ and
 $$\lim_{k\to +\infty} u_k=u\hbox{ in }C^{4}_{loc}(\overline{\Omega}\setminus S).$$
 \begin{proof} Indeed, with \eqref{bnd:L1}, we get that $\Vert \htu_k\Vert_{L^1(\Omega)}\leq
 C$ for all $k\in\mathbb{N}$. It then follows from Theorem 1.2 of
 \cite{robert} that there exists $\htu\in C^4(\Omega)$ such that $\lim_{k\to
 +\infty}\htu_k=\htu$ in $C^3_{loc}(\Omega)$. Therefore $S\subset \partial \Omega$.
 It then follows from Claims 6 and 7 and standard elliptic theory that there exists $u\in C^4(\overline{\Omega}\setminus S)$ such that
$$\lim_{k\to +\infty} u_k=u\hbox{ in }C^{4}_{loc}(\overline{\Omega}\setminus S).$$
 Moreover, passing to the limit $k\to +\infty$ in Claim 7, we get
 that
 $$\inf_{i=1,...,N}|x-x_i||\nabla u(x)|\leq C\hbox{ for all }x\in \Omega\setminus S.$$
We are left with proving that $u$ can be smoothly extended to $S$. We fix $x_0\in S$ and we let $\delta>0$ small enough such that
$$|x-x_0||\nabla u(x)|\leq C\hbox{ for all }x\in \Omega\cap B_\delta(x_0)\setminus\{x_0\}.$$
Therefore, there exists $C'>0$ such that for all $x,y\in
\Omega\cap B_\delta(x_0)\setminus\{x_0\}$ such that
$|x-x_0|=|y-x_0|$, we have that
$$|u(x)-u(y)|\leq C'.$$
Taking $y\in\partial\Omega$, we then get $|u(x)|\leq C'$ for all
$x\in \Omega\cap B_\delta(x_0)\setminus\{x_0\}$. Proceeding
similarly for all the points of $S$, we get that there exists
$C>0$ such that $|u(x)|\leq C$ for all $x\in\Omega\setminus
S$.\par
\medskip\noindent We let $w\in H^2_0(\Omega)$ such that $\Delta^2w=
e^{-\alpha_\infty}e^u$. (Since $|u| \leq C$, we may simply put $e^u=1$ when $ x=x_0$.)  It follows from standard theory that $w\in
C^3(\overline{\Omega})$ and that
$$w(x)=\int_\Omega G(x,y)e^{-\alpha_\infty}e^{u(y)}\, dy$$
for all $x\in\Omega$.
For  $\delta>0$ small enough and $x \in \bar{\Omega} \backslash S$,
\begin{equation}
\label{2n}
 u_k (x)= \int_\Om G(x, y) e^{ \htu_k (y)}\, dy=\int_{  (\bigcup_{i=1}^N B_\delta (x_i))^c \cap \Om} G(x, y) e^{ \htu_k (y)}\, dy + O(\delta).
\end{equation}

Passing to the limit (first in $k$ and then in $\delta$)  in \eqref{2n} and noting that $|u| \leq C$, we get
that
$$u(x)=\int_\Omega G(x,y)e^{-\alpha_\infty}e^{u(y)}\, dy$$
for all $x\in\bar{\Omega} \backslash S$. Therefore, $u\equiv w$ in $\bar{\Omega} \backslash S$ and $u$
can be extended smoothly as a $C^3-$function on
$\overline{\Omega}$. Coming back to the definition of $w$, we get
that $w$ is $C^4$ and then $u\in C^{4}(\overline{\Omega})$. This
ends the proof of Claim 10. As a remark, let us note that if the
concentration points were isolated (that is $x_i\neq x_j$ for all
$i\neq j$), the argument above would prove that $(u_k)$ is bounded
uniformly near the boundary, which would immediately exclude
boundary blow-up.\end{proof}

\medskip\noindent Now, we exclude the boundary blow-up in case $\lim_{k\to
+\infty}\alpha_k=+\infty$:\par
\medskip\noindent{\bf Claim 11:} Assume that $\lim_{k\to +\infty}\alpha_k=+\infty$. Let $x_0\in\partial\Om$. Then
$$\lim_{r \to 0} \lim_{k \to +\infty} \int_{B_r (x_0) \cap \Om}
e^{ \htu_k }\, dx=0.$$ In particular, $S\cap \partial
\Om=\emptyset$.
\begin{proof} We argue by contradiction and we let $x_0 \in \partial\Om\cap S$. Then \eqref{eq:1} yields
$$\lim_{r \to 0} \lim_{k \to +\infty} \int_{B_r (x_0) \cap \Om} e^{
\htu_k }\, dx \geq 64 \pi^2.$$
Thus for all $\delta>0$, we have that
\begin{equation}\label{lim:boundary}
\int_{B_\delta (x_0) \cap \Om} e^{ \htu_k }\, dx \geq 32 \pi^2
\end{equation}
for all $k\in\mathbb{N}$ large enough. Furthermore, we may assume
that $S \cap B_{\delta} (x_0)= \{ x_0\}$. Let $y_k:= x_0 + \rho_{k,
r} \nu (x_0)$ with
\begin{equation}
\rho_{k, r}= \frac{ \int_{\partial \Om \cap  B_{r} (x_0)} (x-x_0, \nu) (\Delta u_k)^2\, dx}{ \int_{\partial \Om \cap B_r (x_0)} (\nu (x_0), \nu) (\Delta u_k)^2\, dx}
\end{equation}
where $r << r_1$ such that $ \frac{1}{2} \leq (\nu (x_0) \cdot \nu) \leq 1$ for $x \in \bar{B}_r(x_0) \cap \Om$. Here $\nu(x)$ is the outer normal vector to $T_{x_0}\partial\Omega$ at $x$. Then it is easy to see that $ |\rho_{k, r}| \leq 2 r$ and
\begin{equation}
\label{6}
 \int_{\partial \Om \cap B_r (x_0)} (x- y_k, \nu) (\Delta u_k)^2 \, dx=0.
\end{equation}
Now applying the Pohozaev's identity in $\Om \cap B_r (x_0)$ with
$ y=y_k$, $f(u) = e^{-\alpha_k} e^{u_k}$ and $F(u)= e^{-\alpha_k} (e^{
u_k}-1)  $,  and using Dirichlet boundary condition and (\ref{6}),
we obtain that
\[   4 \int_{\Om \cap B_r (x_0)} (e^{ \htu_k}- e^{- \alpha_k} )\, dx=
 \int_{ \Om \cap  \partial B_r (x_0)}   \langle x-y_k, \nu \rangle (e^{-\alpha_k} e^{ u_k } - e^{-\alpha_k})\, d\sigma\]
\[
- 2 \int_{\Om \cap \partial B_r (x_0)}  \frac{\partial u_k}{\partial \nu}  \Delta u_k\, d\sigma\]
\[ + \int_{\Om \cap \partial B_r (x_0)} \Bigg[   \frac{1}{2} \langle x-y_k, \nu \rangle  (\Delta u_k)^2
 + \frac{\partial  (-\Delta u_k)}{\partial \nu}  <x-y_k, \nabla u_k> \Bigg]\, d\sigma
\]
\[+ \int_{\Om \cap \partial B_r (x_0)} \Bigg[ -
\frac{\partial}{\partial \nu}  u_k  \langle x-y_k, \nabla \Delta u_k\rangle +
 <\nabla u_k, \nabla \Delta u_k> <x-y_k, \nu>\Bigg]\, d\sigma.
\]
Note that $ u_k \to \Psi =\sum_{i=1}^N \gamma_i  G(x, x_i)$ in
$C^3 (\bar{\Om} \backslash S)$, where $G(x, x_0)=0$. Thus we
obtain that all the terms in the last three integrals are of the
form
\[  \lim_{k\to +\infty}\int_{\Om \cap \partial B_r (x_0)} \Bigg[  O(1) \Bigg]\, dx= O( r^3)\]
while
\[  \lim_{k\to +\infty}\int_{\partial \Om \cap B_r (x_0)}   (x-y_k, \nu) (e^{-\alpha_k} e^{u_k} - e^{- \alpha_k} )\,d\sigma= O(r^4).\]
Since $\lim_{k\to +\infty}\alpha_k=+\infty$, we thus obtain that
\begin{equation}
\left| \int_{\Om \cap B_r (x_0)} e^{\htu_k}\, dx\right|\leq C
r^3\end{equation} for $k\in\mathbb{N}$ large enough. Therefore,
$$\lim_{r\to 0}\lim_{k\to +\infty}\int_{\Om \cap B_r (x_0)} e^{\htu_k}\, dx=0.$$
A contradiction with \eqref{lim:boundary}. This proves Claim 11.
\end{proof}

\medskip\noindent{\bf Claim 12:} We have that
$$\lim_{k\to +\infty}\alpha_k=+\infty.$$
\begin{proof} We argue by contradiction and assume that, up to extracting a subsequence, $\lim_{k\to
+\infty}\alpha_k=\alpha_\infty\in\mathbb{R}$. We let $x_0\in
S\subset \partial\Omega$ (this follows from Claim 10). Arguing as
in Claim 11, we get that
\begin{eqnarray*}
&&4 \int_{\Om \cap B_r (x_0)} (e^{\htu_k}-e^{- \alpha_k}) \, dx=
 \int_{ \Om \cap  \partial B_r (x_0)}   \langle x-y_k, \nu \rangle (e^{-\alpha_k} e^{ u_k } - e^{-\alpha_k})\,  d
 \sigma\\
&&- 2 \int_{\Om \cap \partial B_r (x_0)}  \frac{\partial
u_k}{\partial \nu}  \Delta u_k\, d\sigma \\
&&+ \int_{\Om \cap
\partial B_r (x_0)} \Bigg[   \frac{1}{2} \langle x-y_k, \nu
\rangle  (\Delta u_k)^2
 + \frac{\partial  (-\Delta u_k)}{\partial \nu}  <x-y_k, \nabla u_k> \Bigg]\,
 d\sigma\\
&&+ \int_{\Om \cap \partial B_r (x_0)} \Bigg[ -
\frac{\partial}{\partial \nu}  u_k  \langle x-y_k, \nabla \Delta
u_k\rangle +
 <\nabla u_k, \nabla \Delta u_k> <x-y_k, \nu>\Bigg]\, d\sigma.
\end{eqnarray*}
Letting $k\to +\infty$, we then get with Claim 10 that
\begin{eqnarray}\label{lim:other}
&& 4 \times 32\pi^2 \leq 4\int_{\Om \cap B_r (x_0)} e^{- \alpha_k}
\, dx+
 \int_{ \Om \cap  \partial B_r (x_0)}   \langle x-y_\infty, \nu \rangle ( e^{ u -\alpha_\infty} - e^{-\alpha_\infty})\,  d
 \sigma\nonumber\\
&&- 2 \int_{\Om \cap \partial B_r (x_0)}  \frac{\partial
u}{\partial \nu}  \Delta u\, d\sigma \nonumber\\
&&+ \int_{\Om \cap
\partial B_r (x_0)} \Bigg[   \frac{1}{2} \langle x-y_\infty, \nu
\rangle  (\Delta u)^2
 + \frac{\partial  (-\Delta u)}{\partial \nu}  <x-y_\infty, \nabla u> \Bigg]\,
 d\sigma\nonumber\\
&&+ \int_{\Om \cap \partial B_r (x_0)} \Bigg[ -
\frac{\partial}{\partial \nu}  u  \langle x-y_\infty, \nabla
\Delta u\rangle +
 <\nabla u, \nabla \Delta u> <x-y_\infty, \nu>\Bigg]\,
 d\sigma\nonumber
\end{eqnarray}
for all $r>0$ small enough, where $y_\infty:=\lim_{k\to
+\infty}y_k$ depends on $r$ with $|y_\infty-x_0|\leq 2r$. With
Claim 10, we know that $u\in C^4(\overline{\Omega})$. Passing to
the limit $r\to 0$ above, we get that the RHS goes
to zero. A contradiction. Then $\lim_{k\to
+\infty}\alpha_k=+\infty$, and Claim 12 is proved.\end{proof}

\medskip \noindent {\bf Claim 13:}  $\gamma_i= 64 \pi^2, i=1, ..., N$.
\begin{proof} Since $x_i \in \Om$, the same proof as in Lemma
3.5 of Lin-Wei \cite{w} gives  the claim. We also refer to Druet-Robert
\cite{dr}.\end{proof}

\noindent {\bf Claim 14:}  The identity (\ref{id}) holds.
\begin{proof} The proof is exactly the same as that of Theorem 1.2 of Lin-Wei \cite{w} and as in Druet-Robert \cite{dr}. \end{proof} \medskip\noindent Theorem \ref{main} follows
form Claims 9-14.

\section{ Proof of Theorem \ref{main2}}
\setcounter{equation}{0} By Theorem \ref{main}, there are no
boundary bubbles for (\ref{1.1}). The proof of Theorem \ref{main2}
follows along the lines of Sections 3 and 4 of \cite{lw}: we just
need to change the Navier boundary condition to Dirichlet boundary
condition. Let us sketch the changes. We first choose a good
approximate function:  fix $P \in \Om$ and let
\begin{equation}
\label{2.2}
U_{\ep, P} (x) := \log \frac{\gamma \ep^{4} }{
(\ep^2+|x-P|^2)^{4}},
\end{equation}
where $\gamma:=3\cdot 2^7=384$. We consider the projection of
$U_{\ep, P}$:
\begin{equation}
\label{2.9} \left\{\begin{array}{l} \Delta^2 {\mathcal P}_\Om
U_{\ep, P} - e^{ U_{\ep, P}} =0 \ \mbox{in} \ \Om,\\ {\mathcal
P}_\Om U_{\ep, P}=  \partial_\nu {\mathcal P}_\Om U_{\ep, P} =0 \
\mbox{on} \
\partial \Om.
\end{array} \right.
\end{equation}
Set
\begin{equation}
\label{2.10}
{\mathcal P}_\Om U_{\ep, P} = U_{\ep, P} - \varphi_{\ep, P}
\end{equation}
Then $\varphi_{\ep, P}$ satisfies
\begin{equation}
\label{2.11}
\left\{\begin{array}{l}
\Delta^{2} \varphi_{\ep, P}  =0 \ \mbox{in} \ \Om, \\
  \varphi_{\ep, P}=U_{\ep, P}, \ \partial_\nu \varphi_{\ep, P} =
\partial_\nu U_{\ep,P}  \ \mbox{on} \ \ \partial \Om.
\end{array} \right.
\end{equation}
On $\partial \Om$, we have for $\ep$ sufficiently small
$$U_{\ep, P} (x) = \log (\gamma \ep^{4}) - 8 \log |x-P| -
\frac{4 \ep^2}{ |x-P|^2} + O(\ep^4)$$ uniformly in $C^{4}
(\partial \Om)$. Comparing (\ref{2.11}) with  (\ref{robin}) and
(\ref{2.12}), we have
\begin{eqnarray}
\label{2.13} \varphi_{\ep, P}  =  \log (\gamma \ep^{4}) - 64 \pi^2
R(x, P) - \ep^2 R_1 (x, P) + O(\ep^4), \ \mbox{in} \ \Om.
\end{eqnarray}
We now use ${\mathcal P}_\Om U_{\ep, P}$ as a test function to
compute an upper bound for $ c_{64 \pi^2 }$. Let $Q_0$ be such
that $R(Q_0, Q_0)=\max_{Q \in \Om} R(Q, Q)$.  Similar computations
in [page 799,\cite{lw}] yield
\[ J_{64 \pi^2} [ {\mathcal P}_\Omega U_{\epsilon, Q_0}]= A_0 -\frac{1}{2} (64\pi^2)^2 \max_{P \in \Om} R(P, P)\]
\[  -\frac{\ep^2}{2} \Bigg[ 64 \pi^2 R_1 (Q_0, Q_0) + \frac{ (64\pi^2)^2}{4} \Delta_x R(Q_0, Q_0)\Bigg] + o(\ep^2)\]
where $A_0$ is a generic constant. By our assumption (\ref{con}),
we have
\begin{equation}
\label{upper}
c_{64 \pi^2} <A_0 -\frac{1}{2} (64\pi^2)^2 \max_{P \in \Om} R(P, P).
\end{equation}
On the other hand, let $u_\rho$ be a minimizer of $J_\rho$ for
$\rho < 64 \pi^2$. If  $ u_{\rho}$ blows up as $\rho \to 64
\pi^2$, then a lower bound can be obtained by following exactly
the same computation in \cite{lw}:
\begin{equation}
\label{lower}
c_{64 \pi^2} \geq A_0 -\frac{1}{2} (64\pi^2)^2 \max_{P \in \Om} R(P, P).
\end{equation}
From (\ref{upper}) and (\ref{lower}), we deduce that blow-up does not occur. Then $u_\rho$ is uniformly bounded from above. It then follows from Lemma \ref{lem:cv} that $u_\rho$ converges to a minimizer of $J_{64\pi^2}$ when $\rho\to 64\pi^2$.\par
\medskip\noindent Finally, when $\Om$ is a ball, (without loss of generality, we may take $\Om=B_1 (0)$), by the result of Berchio, Gazzola and Weth \cite{bgw}, $u$ is radially symmetric and strictly decreasing. Here $Q_0=0$. Now, by the so-called Boggio's formula \cite{boggio}, we have
\begin{equation*}
 G(x, y)= \frac{1}{8 \pi^2} \int_1^{\frac{[x,y]}{|x-y|}} \frac{ (v^2-1)}{v^3}\, dv, \ \mbox{ where } \ [x, y]^2 = |x-y|^2 + (1-|x|^2)(1-|y|^2),
\end{equation*}
 for  $  x, y \in B_1(0)$.  Thus
\[ G(x, 0)= \frac{1}{8 \pi^2} \left( \log \frac{1}{|x|} +  \frac{|x|^2}{2} -\frac{1}{2}\right), R(x, 0)= \frac{1}{8 \pi^2} \left(   \frac{|x|^2}{2} -\frac{1}{2}\right),\]
and hence
\[ \Delta_x R(0, 0) = \frac{1}{2 \pi^2}>0.\]
It is easy to compute $R_1 (x, 0)= 4(2-|x|^2)$ and hence
\[ R_1 (0, 0) = 8 >0.\]
This shows that condition (\ref{con}) is satisfied.  Theorem \ref{main2} is thus proved.

\medskip\begin{center}
{\bf ACKNOWLEDGMENTS}
\end{center}

The research of the second author is partially supported by a
Competitive Earmarked Grant from RGC of HK. This work was carried
out while the first author was visiting the Chinese University of
Hong-Kong: he thanks this institution for its support and its
hospitality.

\end{document}